\documentclass[11pt]{amsart}
\usepackage{amssymb}
\usepackage{amsthm}
\usepackage{mathrsfs}
\usepackage[all]{xy}
\newtheorem{thm}{Theorem}[section]
\newtheorem{cor}[thm]{Corollary}
\newtheorem{lem}[thm]{Lemma}
\newtheorem{prop}[thm]{Proposition}
\theoremstyle{definition}
\newtheorem{defn}[thm]{Definition}

\theoremstyle{remark}
\newtheorem{rem}[thm]{Remark}
\numberwithin{equation}{section}

\newcommand{\abs}[1]{\left\vert#1\right\vert}
\newcommand{\set}[1]{\left\{#1\right\}}

\newcommand{\sph}{\mathbb{S}}

\newcommand{\df}{\mathrm{d}}
\newcommand{\B}{{B}}

\newcommand{\Ric}{\mathrm{Ric}}
\newcommand{\scal}{\mathrm{scal}}
\newcommand{\Hess}{\mathrm{Hess}}
\newcommand{\vol}{\mathrm{vol}}
\newcommand{\dive}{\mathrm{div}}

\begin{document}

\title[Bach flat warped product Einstein manifolds]{On Bach flat warped product Einstein manifolds}%
\author{Qiang Chen}
\address{Department of Mathematics, Lehigh University}
\email{qic208@lehigh.edu}

\author{Chenxu He}%
\address{Department of Mathematics, Lehigh University}%
\email{he.chenxu@lehigh.edu}%

\begin{abstract}
In this paper we show that a compact warped product Einstein manifold with vanishing Bach tensor of dimension $n \geq 4$ is either Einstein or a finite quotient of a warped product with $(n-1)$-dimensional Einstein fiber. The fiber has constant curvature if $n=4$.
\end{abstract}

\maketitle
\section{Introduction}

A $\left( \lambda ,n+m\right) $-Einstein manifold $(M^n,g,f)$ is a complete Riemannian manifold with a smooth function $f$ which satisfies the following $(\lambda, n+m)$-Einstein equation:
\begin{equation}\label{eqnWPEf}
\Ric^m_f = \Ric + \Hess f - \frac{1}{m} \df f\otimes \df f = \lambda g.
\end{equation}
When $m$ is a positive integer, $(\lambda, n+m)$-Einstein metrics are exactly those $n$-dimensional manifolds which are the base of an $n+m$ dimensional Einstein warped product, i.e., $(M\times F^m, g + e^{-2f/m} g_F)$ is an Einstein manifold with Einstein constant $\lambda$ where $(F^m, g_F)$ is another Einstein manifold, see \cite{Besse}.

$\Ric^m_f$ in equation (\ref{eqnWPEf}) is also called the \emph{m-Bakry Emery tensor}. Lower bounds on this tensor are related to various comparison theorems for the measure $e^{-f} \df \mathrm{vol}_g$, see for example Part II of \cite{Villani}, \cite{WeiWylie} and the references therein. From these comparison theorems, the $(\lambda, n+m)$-Einstein equation is the natural Einstein condition of having constant $m$-Bakry-Emery Ricci tensor. $(\lambda, n+1)$-Einstein metrics are more commonly called \emph{static metrics} and such metrics have been extensively studied for their connections to scalar curvature, the positive mass theorem and general relativity. We do not consider these metrics here and we assume $m \ne 1$ throughout this paper. Taking $m \rightarrow \infty$, one also obtains the gradient Ricci soliton equation
\begin{eqnarray*}
\mathrm{Ric} + \mathrm{Hess }f = \lambda g.
\end{eqnarray*}
We could then also call a gradient Ricci soliton a $(\lambda, \infty)$-Einstein manifold. Ricci solitons have been studied because of their connection to Ricci flow and that they are natural generalization of Einstein manifolds. We refer to the survey paper\cite{Caosoliton} and references therein for recent progress on this subject.

In a series of papers with P. Petersen and W. Wylie, the second author studied warped product Einstein manifolds under various curvature and symmetry conditions, see \cite{HPWlcf,HPWrigidity,HPWvirtual}. Many interesting results on gradient Ricci solitons are also obtained on warped product Einstein manifolds. However we also found some nontrivial examples, i.e., not Einstein or product of them, on homogeneous spaces in \cite{HPWrigidity}. Those examples are in stark contrast to the gradient Ricci soliton case, where all homogeneous gradient Ricci solitons are Einstein or product of Einstein manifolds.

In this paper we consider an interesting class of complete warped product Einstein manifolds: those with vanishing Bach tensor. This well-known tensor was first introduced by R. Bach in \cite{Bach} to study conformal relativity in early 1920s'. On any Riemannian manifold $(M^n, g)(n\geq 4)$ the Bach tensor is defined by
\begin{equation*}
\B(X,Y) = \frac{1}{n-3}(\nabla^2_{E_i, E_j}W)(X,E_i, E_j, Y) + \frac{1}{n-2}\Ric(E_i,E_j)W(X, E_i,E_j, Y).
\end{equation*}
Here $\set{E_i}_{i=1}^n$ is an orthonormal frame, $\nabla^2_{E_i, E_j}$ is the covariant derivative of tensors and $W$ is the Weyl curvature tensor. In the case when the manifold is Einstein or locally conformal flat, the Bach tensor vanishes. The dimension 4 is most interesting since on any compact 4-manifold $(M^4, g)$, Bach flat metrics are precisely the critical points of the following conformally invariant functional on the space of metrics:
\begin{equation*}
\mathcal{W}(g) = \int_M |W_g|^2 d\vol_g,
\end{equation*}
where $W_g$ is the Weyl tensor of the metric $g$. Other than Einstein and locally conformally flat metrics, there are two more classes of compact 4-manifolds with vanishing Bach tensor: metrics are locally conformal to an Einstein one, and half conformally flat metrics (self-dual or anti-self-dual) if $M^4$ is orientable. The aim of this paper is to show that a stronger converse on warped product Einstein metrics holds. The proof is motivated by a recent corresponding result on gradient Ricci solitons in \cite{caochenBach} proved by H.-D. Cao and the first author.

\begin{thm}\label{thmWPEBachflatD4}
Suppose $(M^4,g,f)$ is a compact $(\lambda, 4+m)$-Einstein manifold with $m \ne 0, 1$ or $2-n$. We assume further that the Bach tensor $\B$ of $M$ vanishes everywhere. Then $M$ is locally conformal flat.
\end{thm}

Theorem \ref{thmWPEBachflatD4} is a direct consequence of the following more general result combined with some early results in \cite{HPWlcf}.
\begin{thm}\label{thmWPEBachflatallD}
Suppose $(M^n,g,f)(n \geq 4)$ is a compact $(\lambda, n+m)$-Einstein manifold with $m \ne 0, 1$ or $2-n$. We assume further that the Bach tensor $\B$ of $M$ vanishes everywhere. Then $M$ has harmonic Weyl tensor and $W(X,Y,Z, \nabla f) = 0$ for any vector fields $X$, $Y$ and $Z$.
\end{thm}

\begin{rem}
Note that our Theorem \ref{thmWPEBachflatallD} is analogous to Theorem 5.1 in \cite{caochenBach}.
\end{rem}
\begin{rem}
In the case when $m=2-n$, our argument breaks down since in one key identity, the equation (\ref{eqnCDWeyl}), some coefficient vanishes. On the other hand, it is observed in \cite{CMMR} that in this case a $(\lambda, n+(2-n))$-Einstein metric is globally conformal Einstein. In particular it has vanishing Bach tensor when $n=4$.
\end{rem}
\begin{rem}
In \cite{Bohm}, C. B\"{o}hm constructed compact rotationally symmetric $(\lambda, n+m)$-Einstein metrics on $\sph^n$ for $n=3, 4, 5, 6, 7$ which are not Einstein. This is in sharp contrast to the gradient Ricci solitons. These examples also show that our conclusion in dimension 4 cannot be strengthened.
\end{rem}
\begin{rem}
Theorem \ref{thmWPEBachflatD4} was first obtained by G. Catino in \cite{Catinohalfweyl} under a stronger assumption that $(M^4, g)$ is half conformally flat.
\end{rem}

If $m$ is positive then from the comparison theorems of $m$-Bakry Emery tensor in \cite{Qian}, a $(\lambda, n+m)$-Einstein manifold is compact if and only if $\lambda > 0$. Using Theorem 1.5 in \cite{HPWlcf}, the global classification of warped product Einstein manifolds with harmonic Weyl tensor and $W(\nabla f, \cdot, \cdot, \nabla f) = 0$, Theorem \ref{thmWPEBachflatallD} has the following
\begin{cor}
Let $m \ne 1$ be a positive number. Suppose that $(M^n,g, f)(n\geq 4)$ is a simply-connected $(\lambda, n+m)$-Einstein manifold with $\lambda > 0$ and has vanishing Bach tensor. Then $(M^n, g, f)$ is either
\begin{enumerate}
\item Einstein with constant function $f$, or
\item $g = \df t^2 + \psi^2(t)g_L$, $f = f(t)$, where $g_L$ is Einstein with non-negative Ricci curvature, and has constant curvature if $n=4$.
\end{enumerate}
\end{cor}
\begin{rem}
In the proof of Theorem 1.5 in \cite{HPWlcf} the authors made the assumption that $m >1$. In fact the whole argument carries over the case when $0 < m < 1$.
\end{rem}

\smallskip

The paper is organized as follows. In section 2 we recall definitions and basic properties of Bach, Cotton, Weyl tensors and, the $D$-tensor defined in \cite{caochenlcf}. Note that in \cite{caochenlcf}, the $D$-tensor is denoted by $B$. Here $B$ is referred to the Bach tensor. We also list some relevant properties of warped product Einstein metrics. In section 3, we show how the $D$-tensor characterizes the geometry of the level set of $f$, see Proposition \ref{propWPEvanishingD}. In section 4, we prove Theorem \ref{thmWPEBachflatD4} and Theorem \ref{thmWPEBachflatallD}.

\medskip{}

\textbf{Acknowledgment:} The authors would like to thank Huai-Dong Cao for enlightening conversations and helpful suggestions.

\medskip{}

\section{Preliminaries}
In this section we set up our notations and recall some well-known facts on warped product Einstein manifolds. For more detail, see for example \cite{HPWlcf} and references therein.

We use the convention that the Riemann curvature tensor $R(X,Y,Y,X)$ has the same sign as the sectional curvature of the 2-plane spanned by $X$ and $Y$. For $n\geq 4$ the Weyl curvature tensor is defined as
\begin{equation*}
R = W + \frac{2}{n-2}\Ric \odot g - \frac{\scal}{(n-1)(n-2)}g\odot g,
\end{equation*}
where, for two symmetric (0,2)-tensors $s$ and $r$, we define the
Kulkarni-Nomizu product $s\odot r$ to be the $(0,4)$-tensor
\begin{eqnarray*}
(s\odot r)(X,Y,Z,W) & = & \frac{1}{2}\left( r(X,W)s(Y,Z)+r(Y,Z)s(X,W) \right.\\
& & \left. -r(X,Z)s(Y,W)-r(Y,W)s(X,Z)\right).
\end{eqnarray*}
Recall that for any $X, Y \in TM$ the Bach tensor $\B$ is the symmetric (0,2)-tensor defined by
\begin{equation}\label{eqnBachtensor}
\B(X,Y) = \frac{1}{n-3}\sum_{i,j} (\nabla^2_{E_i, E_j}W)(X,E_i,E_j,Y) + \frac{1}{n-2}\sum_{i,j}\Ric(E_i,E_j)W(X,E_i,E_j,Y),
\end{equation}
where $\set{E_i}_{i=1}^n$ is an orthonormal frame and $\nabla^2_{E_i,E_j} W$ is the covariant derivative of the Weyl tensor.

The Schouten tensor is the $(0,2)$-tensor
\begin{equation*}
S = \Ric - \frac{\scal}{2(n-1)}g
\end{equation*}
and the Cotton tensor $C$ is defined as
\begin{equation*}
C(X,Y,Z) =  \left(\nabla_X S\right)(Y,Z) - \left(\nabla_Y S\right)(X,Z), \quad \mbox{for any } X,Y, Z \in TM.
\end{equation*}
Using the fact that $(\dive R)(X,Y,Z) = (\nabla_X \Ric)(Y,Z) - (\nabla_Y \Ric)(X,Z)$ we have
\begin{equation}\label{eqnCottontensor}
C(X,Y,Z) = (\dive R)(X,Y,Z) - \frac{1}{2(n-1)}\left((\nabla_X \scal)g(Y,Z) - (\nabla_Y \scal)g(X,Z)\right).
\end{equation}

\begin{defn}
A Riemannian manifold $(M^n,g)$ has \emph{harmonic Weyl tensor} if the Cotton tensor vanishes.
\end{defn}
\begin{rem}
For $n\geq 4$ the Cotton tensor is, up to a constant factor, the divergence of the Weyl tensor:
\begin{equation}\label{eqnCottondivWeyl}
C(X,Y,Z) = \frac{n-2}{n-3}(\dive W)(X,Y,Z), \quad \mbox{ for any } X, Y, Z \in TM.
\end{equation}
So we can rewrite the Bach tensor as
\begin{equation}\label{eqnBachCWeyl}
\B(X,Y) = \frac{1}{n-2}\left(\sum_i (\nabla_{E_i} C)(E_i, X, Y) + \sum_{i,j}\Ric(E_i, E_j)W(X, E_i, E_j, Y)\right)
\end{equation}
where $\set{E_i}_{i=1}^n$ is an orthonormal frame.
\end{rem}
\begin{rem}
Note that if $n =3$, $W = 0$ and then harmonic Weyl tensor is equivalent to $(M^3, g)$ being locally conformal flat. If $n \geq 4$ then $M$ has harmonic Weyl tensor if and only if $\dive W = 0$, and $M$ is locally conformal flat if and only if $W = 0$.
\end{rem}

On a $(\lambda, n+m)$-Einstein manifold $(M,g, f)$ for any $X, Y, Z \in TM$ we define the $D$-tensor, which is identical to the one in \cite{caochenlcf} (and \cite{caochenBach}) for Ricci solitons, as follows:
\begin{eqnarray}
D(X,Y,Z) & = &  \frac{1}{(n-1)(n-2)}\left(\Ric(X, \nabla f) g(Y,Z) - \Ric(Y, \nabla f) g(X,Z)\right) \notag \\
& & + \frac{1}{n-2}\left(\Ric(Y,Z)g(X, \nabla f) - \Ric(X,Z)g(Y, \nabla f)\right) \label{eqnbetatensor} \\
& & - \frac{\scal}{(n-1)(n-2)}\left(g(X, \nabla f)g(Y,Z) - g(Y, \nabla f) g(X,Z)\right). \notag
\end{eqnarray}

Note that both $C$ and $D$ tensors are skew-symmetric in their first two indices and trace-free in any two indices:
\begin{equation*}
C(X, Y, Z) = - C(Y,X,Z), \quad \sum_{i}C(E_i, E_i, X) = \sum_{i}C(E_i, X, E_i) = 0;
\end{equation*}
and
\begin{equation*}
D(X, Y, Z) = - D(Y,X,Z), \quad \sum_{i}D(E_i, E_i, X) = \sum_{i}D(E_i, X, E_i) = 0.
\end{equation*}

Next we recall some properties on warped product Einstein manifolds and the proofs can be found in \cite{HPWlcf}. The function $\rho$ is defined by
\begin{equation*}
\scal = (n-1)\lambda - (m-1)\rho.
\end{equation*}
Note that a $(\lambda, n+1)$-Einstein manifold has constant scalar curvature $(n-1)\lambda$. The modified Ricci and Riemann curvature tensors are defined by
\begin{equation*}
P = \Ric + \rho g,
\end{equation*}
and
\begin{eqnarray*}
Q & = & R + \frac{2}{m}\Ric \odot g - \frac{\lambda + \rho}{m}g\odot g \\
& = & R + \frac{2}{m}P \odot g + \frac{\rho - \lambda}{m}g\odot g.
\end{eqnarray*}

\begin{prop}\label{propWPEPQ}
Suppose $(M, g, f)$ is a $(\lambda, n+m)$-Einstein manifold with $m \ne 1$. Then we have
\begin{eqnarray}
& & P(\nabla f) = - \frac{m}{2}\nabla \rho, \mbox{ or equivalently } \Ric(\nabla f) = - \frac{m}{2}\nabla \rho + \rho \nabla f; \label{eqnWPEPf} \\
& & (\dive R) = Q(X,Y,Z, \nabla f) - \frac{1}{m}(g\odot g)(X,Y,Z,P(\nabla f)). \label{eqnWPEdivR}
\end{eqnarray}
\end{prop}
The first equation (\ref{eqnWPEPf}) was proved in \cite[(3.12)]{CSW} and the second one (\ref{eqnWPEdivR}) was shown in \cite[Porposition 6.3]{HPWlcf}.

\medskip{}
\section{The covariant $3$-tensor $D$}

In this section we extend some known results of the $3$-tensor $D$ from gradient Ricci solitons to warped product Einstein manifolds. Since the $(\lambda, n+m)$-Einstein equation (\ref{eqnWPEf}) contains extra term $-\frac{1}{m} \df f\otimes \df f$ we provide the calculations in detail though we essentially follow proofs in \cite{caochenlcf} and \cite{caochenBach}.

On gradient Ricci solitons, the $D$ tensor relates the Cotton tensor and Weyl tensor in the following way, see \cite[Lemma 3.1]{caochenBach}:
\begin{equation*}
C(X,Y,Z) = D(X,Y,Z) + W(X,Y,Z, \nabla f), \quad \mbox{for any } X, Y, Z \in TM.
\end{equation*}
On warped product manifolds we have the similar relation for these three tensors.
\begin{lem}\label{lemCDWeyl}
Suppose $(M^n,g,f)$ is a $(\lambda, n+m)$-Einstein manifold, then the Cotton tenor $C$, $D$-tensor and Weyl tensor $W$ satisfy the following identity:
\begin{equation}\label{eqnCDWeyl}
C(X,Y,Z) = W(X,Y,Z, \nabla f) + \frac{m+n-2}{m}D(X,Y,Z), \quad \mbox{for any }X, Y, Z \in TM.
\end{equation}
\end{lem}
\begin{rem}
The above identity has been observed by G. Catino in \cite{Catinohalfweyl} when $n=4$.
\end{rem}

\begin{proof}
From the formula (\ref{eqnWPEdivR}) of $\dive R$, the definition of $Q$-tensor and the decomposition curvature tensor $R$, we have
\begin{eqnarray*}
& & (\dive R)(X,Y,Z) \\
& = & Q(X,Y,Z, \nabla f) - \frac{1}{m}(g\odot g)(X,Y,Z,P(\nabla f)) \\
& = & R(X,Y,Z, \nabla f) + \frac{2}{m} (\Ric \odot g)(X,Y,Z, \nabla f) - \frac{\lambda + \rho}{m}(g\odot g)(X,Y,Z, \nabla f) \\
& & - \frac{1}{m}(g\odot g)(X, Y, Z, P(\nabla f)) \\
& = & W(X,Y, Z, \nabla f) \\
& & + \frac{2}{n-2}(\Ric \odot g)(X,Y,Z, \nabla f) - \frac{(n-1)\lambda - (m-1)\rho}{(n-1)(n-2)}(g\odot g)(X,Y, Z, \nabla f) \\
& & + \frac{2}{m} (\Ric \odot g)(X,Y,Z, \nabla f) - \frac{\lambda + \rho}{m}(g\odot g)(X,Y,Z, \nabla f) - \frac{1}{m}(g\odot g)(X, Y, Z, P(\nabla f)) \\
& = & W(X,Y,Z, \nabla f) \\
& &  - \frac{1}{m}(g\odot g)(X,Y,Z,P(\nabla f)) + \frac{2(m+n-2)}{m(n-2)}(\Ric\odot g)(X,Y, Z, \nabla f) \\
& &- \frac{(n-1)(m+n-2)\lambda + ((n-1)(n-2) -m(m-1))\rho}{m(n-1)(n-2)}(g\odot g)(X,Y,Z, \nabla f).
\end{eqnarray*}
Using the fact that $P = \Ric - \rho g$ we have
\begin{eqnarray*}
& & (\dive R)(X,Y,Z) \\
& = & W(X,Y,Z, \nabla f) \\
& & + \frac{1}{n-2}\left(\Ric(X, \nabla f)g(Y,Z) - \Ric(Y, \nabla f)g(X,Z)\right) \\
& & + \frac{m+n-2}{m(n-2)}\left(\Ric(Y,Z)g(X, \nabla f) - \Ric(X,Z)g(Y, \nabla f)\right) \\
& & - \frac{(n-1)(m+n-2)\lambda - m(m-1)\rho}{m(n-1)(n-2)}\left(g(X, \nabla f)g(Y,Z) - g(Y, \nabla f) g(X,Z)\right).
\end{eqnarray*}
From the formula (\ref{eqnWPEPf}) of $\Ric(\nabla f)$ we have
\begin{eqnarray*}
& & (\dive R)(X,Y,Z, \nabla f) - W(X,Y,Z, \nabla f) \\
& = & - \frac{m}{2(n-2)}\left((\nabla_X \rho)g(Y,Z) - (\nabla_Y \rho)g(X,Z)\right) \\
& & + \frac{m+n-2}{m(n-2)}\left(\Ric(Y,Z)g(X, \nabla f) - \Ric(X,Z)g(Y, \nabla f)\right) \\
& & - \frac{m+n-2}{m(n-1)(n-2)}\left((n-1)\lambda - m \rho\right)\left(g(X, \nabla f)g(Y,Z) - g(Y, \nabla f) g(X,Z)\right)
\end{eqnarray*}
From the defining equation (\ref{eqnCottontensor}) of the Cotton tensor $C$ and $\scal = (n-1)\lambda - (m-1)\rho$ we have
\begin{equation*}
C(X,Y,Z) = (\dive R)(X,Y,Z) + \frac{m-1}{2(n-1)}\left((\nabla_X \rho) g(Y,Z) - (\nabla_Y \rho) g(X,Z)\right)
\end{equation*}
and then
\begin{eqnarray*}
& & \frac{m}{m+n-2}(C(X,Y,Z) - W(X,Y,Z, \nabla f)) \\
& = &  - \frac{m}{2(n-1)(n-2)}\left((\nabla_X \rho)g(Y,Z) - (\nabla_Y \rho)g(X,Z)\right) \\
& & + \frac{1}{n-2}\left(\Ric(Y,Z)g(X, \nabla f) - \Ric(X,Z)g(Y, \nabla f)\right) \\
& & - \frac{(n-1)\lambda - m \rho}{(n-1)(n-2)}\left(g(X, \nabla f)g(Y,Z) - g(Y, \nabla f) g(X,Z)\right)
\end{eqnarray*}
which is exactly equal to $D(X,Y,Z)$ by the formula of $\Ric(\nabla f)$.
\end{proof}

On gradient Ricci solitons, one amazing fact of $D$ tensor is that its norm is linked to the geometry of the level set of the potential function $f$, see \cite[(4.5)]{caochenlcf} and [Lemma 3.2]\cite{caochenBach}. We have the following extension to warped product Einstein manifolds.

\begin{lem}\label{lemBnormlevelset}
Suppose $(M^n,g, f)$ be a $(\lambda, n+m)$-Einstein manifold. Let $\Sigma^{n-1}$ be a level set of $f$ with $\nabla f(p) \ne 0$ and let $h_{ab}$($a,b =2, \ldots, n$) and $H = (n-1)\sigma$ be its second fundamental form and mean curvature respectively. Then we have
\begin{equation}\label{eqnBnormlevelset}
\abs{D}^2 = \frac{2|\nabla f|^4}{(n-2)^2}\sum_{a,b=2}^n \abs{h_{ab} - \sigma g_{ab}}^2 + \frac{m^2}{2(n-1)(n-2)(m-1)^2}\abs{\nabla^{\Sigma}\scal}^2,
\end{equation}
where $\abs{\nabla^{\Sigma}\scal}^2 = |\nabla \scal|^2 - \left(\nabla \scal\cdot \frac{\nabla f}{\abs{\nabla f}}\right)^2$.
\end{lem}
\begin{proof}
Let $\set{e_i}_{i=1}^n$ be an orthonormal frame with $e_1 = \frac{\nabla f}{|\nabla f|}$ at the point $\nabla f \ne 0$. The second fundamental form $h_{ab}$ and the mean curvature $H$ of the level hypersurface $\Sigma$ are given by
\begin{eqnarray*}
h_{ab} & = & g\left(\nabla_{e_a}\frac{\nabla f}{\abs{\nabla f}}, e_b\right) = \frac{1}{\abs{\nabla f}} \nabla_{e_a}\nabla_{e_b} f \\
& = & \frac{1}{\abs{\nabla f}}\left(\lambda g_{ab} - \Ric(e_a, e_b)\right) \\
H &= & \frac{1}{\abs{\nabla f}}\left((n-1)\lambda - \scal + \Ric(e_1, e_1)\right).
\end{eqnarray*}
So we have
\begin{eqnarray*}
\sum_{a,b=2}^n |h_{ab}|^2 & = & \frac{1}{|\nabla f|^2}\sum_{a,b=2}^n |\lambda g_{ab} - \Ric(e_a, e_b)|^2 \\
& = & \frac{1}{|\nabla f|^2}\left((n-1)\lambda^2 - 2 \lambda (\scal - \Ric(e_1, e_1)) + \sum_{a,b =2}^n |\Ric(e_a, e_b)|^2\right), \\
H^2 & = & \frac{1}{|\nabla f|^2}\left((n-1)^2 \lambda^2 - 2(n-1)\lambda (\scal - \Ric(e_1, e_1) + (\scal - \Ric(e_1, e_1))^2)\right).
\end{eqnarray*}
From $\Ric(\nabla f) = \rho \nabla f - \frac{m}{2}\nabla \rho$ it follows that
\begin{eqnarray*}
R_{11} = \Ric(e_1, e_1) & = & \rho - \frac{m}{2|\nabla f|^2} \nabla \rho \cdot \nabla f, \\
R_{1a} = \Ric(e_1, e_a) & = & - \frac{m}{2|\nabla f|}\nabla_a \rho.
\end{eqnarray*}
So we have
\begin{equation*}
\sum_{a,b=2}^n \left|h_{ab} - \sigma g_{ab}\right|^2 = \frac{1}{|\nabla f|^2}\abs{\Ric}^2 - \frac{2}{|\nabla f|^2}\sum_{a=2}^n R_{1a}^2 - \frac{1}{|\nabla f|^2}R_{11}^2 - \frac{1}{(n-1)|\nabla f|^2}(\scal - R_{11})^2,
\end{equation*}
where
\begin{eqnarray*}
- \frac{2}{|\nabla f|^2}\sum_{a=2}^n R_{1a}^2 & = & -\frac{m^2}{2|\nabla f|^4}|\nabla \rho|^2 + \frac{m^2}{2|\nabla f|^6}(\nabla \rho \cdot \nabla f)^2, \\
- \frac{1}{|\nabla f|^2}R^2_{11} & = & - \frac{1}{|\nabla f|^2}\rho^2 + \frac{m \rho}{|\nabla f|^4}\nabla \rho \cdot \nabla f - \frac{m^2}{4|\nabla f|^6}(\nabla \rho\cdot \nabla f)^2, \\
\scal - R_{11} & = & (n-1)\lambda - m \rho + \frac{m}{2|\nabla f|^2}\nabla \rho \cdot \nabla f, \\
- \frac{1}{(n-1)|\nabla f|^2}(\scal - R_{11})^2 & = & - \frac{(n-1)\lambda^2}{|\nabla f|^2} - \frac{m^2 \rho^2}{(n-1)|\nabla f|^2} + \frac{2m \lambda \rho}{|\nabla f|^2}, \\
& & - \frac{m^2}{4(n-1)|\nabla f|^6}(\nabla \rho \cdot \nabla f)^2 + \frac{m(m\rho - (n-1)\lambda)}{(n-1)|\nabla f|^4} \nabla \rho \cdot \nabla f.
\end{eqnarray*}
Adding them together yields
\begin{eqnarray*}
\sum_{a,b=2}^n \left|h_{ab} - \sigma g_{ab}\right|^2 & = & \frac{1}{|\nabla f|^2}\abs{\Ric}^2 - \frac{m^2}{2|\nabla f|^4}\abs{\nabla \rho}^2 \\
& & + \frac{m^2(n-2)}{4(n-1)|\nabla f|^6}(\nabla \rho \cdot \nabla f)^2 + \frac{m((m+n-1)\rho - (n-1)\lambda)}{(n-1)|\nabla f|^4}\nabla \rho \cdot \nabla f \\
& & - \frac{m^2+n-1}{(n-1)|\nabla f|^2}\rho^2 + \frac{2m}{|\nabla f|^2}\lambda \rho - \frac{n-1}{|\nabla f|^2}\lambda^2.
\end{eqnarray*}

Let $D_{ijk} = D(e_i, e_j,e_k)$, then we have
\begin{equation*}
D_{ijk} = b_1\left(\nabla_i \rho \delta_{jk} - \nabla_j \rho \delta_{ik}\right) + b_2\left(\nabla_i f R_{jk} - \nabla_j f R_{ik}\right)+ b_3\left(\nabla_i f \delta_{jk} - \nabla_j f\delta_{ik}\right)
\end{equation*}
where $\nabla_i = \nabla_{e_i}$ and
\begin{equation*}
b_1 = - \frac{m}{2(n-1)(n-2)}, \quad b_2 = \frac{1}{n-2}, \quad b_3 =  - \frac{(n-1)\lambda - m \rho}{(n-1)(n-2)}.
\end{equation*}
So we have
\begin{eqnarray*}
|D|^2 & = & \sum_{i,j,k = 1}^n D_{ijk}^2 \\
& = & b_1^2 (2(n-1)|\nabla \rho|^2) + b_2^2 \left(2|\nabla f|^2 \abs{\Ric}^2 - 2\Ric^2(\nabla f, \nabla f)\right) + b_3^2 \left( 2(n-1)|\nabla f|^2 \right) \\
& & + 2b_1b_2\left(2\scal \nabla \rho \cdot \nabla f - 2 \Ric(\nabla f, \nabla \rho)\right) + 2b_1 b_3\left(2(n-1)\nabla \rho \cdot \nabla f\right) \\
& & + 2b_2 b_3\left(2|\nabla f|^2 \scal - 2\Ric(\nabla f, \nabla f)\right) \\
& = & \frac{1}{2(n-1)(n-2)^2}\left(4(n-1)|\nabla f|^2\abs{\Ric}^2 - m^2 n |\nabla \rho|^2\right) \\
& & + \frac{4m\left((m+n-1)\rho - (n-1)\lambda\right)}{2(n-1)(n-2)^2} (\nabla \rho \cdot \nabla f) - \frac{4\left(((n-1)\lambda - m \rho)^2+(n-1)\rho^2\right)}{2(n-1)(n-2)^2}|\nabla f|^2.
\end{eqnarray*}
A straightforward computation shows that
\begin{equation*}
\abs{D}^2 = \frac{2|\nabla f|^4}{(n-2)^2}\sum_{a,b=2}^n \abs{h_{ab} - \sigma g_{ab}}^2 + \frac{m^2}{2(n-1)(n-2)}\abs{\nabla^{\Sigma}\rho}^2.
\end{equation*}
Substituting the function $\rho$ by $\scal$ gives us the desired identity in this lemma.
\end{proof}

Similarly the vanishing of $D$ tensor implies many nice properties about the geometry of the warped product Einstein manifold $(M^n, g,f)$ and the level sets of $f$.
\begin{prop}\label{propWPEvanishingD}
Suppose ($M^n,g, f$)($n \geq 3$) is a $(\lambda,n+m)$-Einstein manifold with $m\ne 1$ and $D=0$. Let $c$ be a regular value of $f$ and $\Sigma = \set{ x\in M | f(x) = c}$ be the level hypersurface of $f$. Then we have
\begin{enumerate}
\item both the scalar curvature and $|\nabla f|^2$ are constant on $\Sigma$;
\item on $\Sigma$, the Ricci tensor either has either a unique eigenvalue or, two distinct eigenvalues with multiplicity $1$ and $n-1$, moreover the eigenvalue with multiplicity $1$ is in the direction of $\nabla f$;
\item the second fundamental form $h_{ab}$ of $\Sigma$ is of the form $h_{ab}=\frac{H}{n-1}g_{ab}$;
\item the mean curvature $H$ is constant on $\Sigma$;
\item $R(\nabla f, X,Y,Z)=0$ for any vectors $X, Y, Z$ tangent to $\Sigma$.
\end{enumerate}
\end{prop}
\begin{proof}
It follows the argument in the proof of \cite[Proposition 3.1]{caochenBach} by using Lemma \ref{lemBnormlevelset}.
\end{proof}
\begin{rem}
If a $(\lambda, n+m)$-Einstein manifold with $m \ne 2-n$ has harmonic Weyl tensor and $W(\nabla f, \cdot, \cdot, \cdot)$, then the $D$ tensor vanishes by Lemma \ref{lemCDWeyl}. So Proposition \ref{propWPEvanishingD} offers an alternative proof of Theorem 7.9 in \cite{HPWlcf} which is the main step for the global classification in Theorem 7.10.
\end{rem}

\medskip{}
\section{The proof of Theorem \ref{thmWPEBachflatD4} and Theorem \ref{thmWPEBachflatallD}}

In this section we first prove Theorem \ref{thmWPEBachflatallD}, i.e., a compact Bach flat $(\lambda, n+m)$-Einstein manifold with $m \ne 0, 1$ or $2-n$ has harmonic Weyl tensor and $W(X,Y,Z, \nabla f) = 0$ for any $X, Y, Z \in TM$. Then Theorem \ref{thmWPEBachflatD4} follows by using Theorem 7.9 in \cite{HPWlcf}.

\begin{proof}[Proof of Theorem \ref{thmWPEBachflatallD}]
We follow the argument in \cite{caochenBach}. Fix a point $p \in M$ and assume that $\set{E_i}_{i=1}^n$ is an orthonormal frame with $\nabla E_i(p) = 0$. Using equation (\ref{eqnCottondivWeyl}), equation (\ref{eqnBachCWeyl}) of Bach tensor and Lemma \ref{lemCDWeyl}, a direct computation shows that for any $X, Y \in TM$ we have
\begin{eqnarray*}
(n-2)\B(X, Y) & = & \sum_i (\nabla_{E_i} C)(E_i, X, Y) + \sum_{i,j}\Ric(E_i, E_j)W(X, E_i, E_j, Y) \\
& = & (\nabla_{E_i} W)(E_i, X, Y, \nabla f) + W(E_i, X, Y, \nabla_{E_i}\nabla f) \\
& & + \frac{m+n-2}{m}(\nabla_{E_i} D)(E_i, X, Y) + \Ric(E_i, E_j)W(X, E_i, E_j, Y) \\
& = & (\dive W)(\nabla f, Y, X) + \frac{m+n-2}{m}(\nabla_{E_i} D)(E_i, X, Y) \\
& & + W(X, E_i, E_j, Y)\left(\Ric(E_i, E_j) + \Hess f(E_i, E_j)\right) \\
& = & \frac{n-3}{n-2}C(\nabla f, Y, X) + \frac{m+n-2}{m}(\nabla_{E_i} D)(E_i, X, Y) \\
& & + W(X, E_i, E_j, Y)\left(\frac{1}{m}g(\nabla f, E_i)g(\nabla f, E_j) + \lambda g(E_i, E_j)\right) \\
& = & \frac{n-3}{n-2}C(\nabla f, Y, X) + \frac{m+n-2}{m}(\nabla_{E_i} D)(E_i, X, Y) \\
& & + \frac{1}{m}W(\nabla f, X, Y, \nabla f).
\end{eqnarray*}
Letting $X = Y = \nabla f$ and integrating on $M$ yield
\begin{eqnarray}
\frac{m(n-2)}{m+n-2}\int_M \B(\nabla f, \nabla f) d\vol & = & \int_M \sum_i(\nabla_{E_i}D)(E_i, \nabla f, \nabla f) d\vol \notag \\
& = & - \int_M \sum_i D(E_i, \nabla f, \nabla_{E_i}\nabla f)d\vol.  \label{eqnIBPBachff}
\end{eqnarray}
For the integrand using the fact that $D$ tensor is trace free for any two indices, we have
\begin{eqnarray*}
- \sum_i D(E_i, \nabla f, \nabla_{E_i}\nabla f)& = & \sum_{i,j} D(E_i, \nabla f, E_j)\left(\Ric(E_i, E_j) - \frac{1}{m}g({E_i}, \nabla f) g({E_j}, \nabla f) - \lambda g(E_i, E_j)\right) \\
& = & \sum_{i,j,k} D(E_i, E_k, E_j)\Ric(E_i, E_j)g(E_k, \nabla f) \\
& = & \frac{1}{2}\sum_{i,j,k} D(E_i, E_k, E_j)\left(\Ric(E_i, E_j)g(E_k, \nabla f) - \Ric(E_k, E_j)g(E_i, \nabla f)\right) \\
& = & - \frac{1}{2}\sum_{i,j,k}\abs{D(E_i, E_j,E_k)}^2.
\end{eqnarray*}
It follows that
\begin{equation}\label{eqnintegralBachff}
\frac{m(n-2)}{m+n-2}\int_M \B(\nabla f, \nabla f) d\vol = - \frac{1}{2}\int_M \abs{D}^2 d\vol.
\end{equation}
So vanishing Bach tensor implies that $D$ tensor vanishes on $M$.

From equation (\ref{eqnCDWeyl}) we have $C(X,Y,Z) = W(X,Y,Z, \nabla f)$. We show that both are zero on the regular points of $f$ and then on $M$ since $f$ is an analytic function, see \cite[Proposition 2.8]{HPWlcf}. At a regular point of $f$ we choose $E_1 = \frac{\nabla f}{\abs{\nabla f}}$ and let $C_{ijk} = C(E_i, E_j, E_k)$. By the symmetry of Weyl tensor we have $C_{ij1} = 0$. Let $a,b, c \geq 2$ be integers. From Proposition \ref{propWPEvanishingD} we have $\Ric(E_1, E_a) = 0$, $R(E_1, E_a, E_b, E_c) = 0$ and thus $W(E_a, E_b, E_c, E_1) = R(E_a, E_b, E_c, E_1) = 0$. So we have $C_{abc} = W(E_a, E_b, E_c, \nabla f) = 0$. It remains to show $C_{1ij} = 0$ for any $i,j =1, \ldots, n$. Since $D = 0$, Bach flatness implies that
\begin{eqnarray*}
0 = (n-2)\B(E_i, E_j) & = & \frac{n-3}{n-2}C_{1ij}\abs{\nabla f} + \frac{1}{m}W(E_1, E_i,E_j,\nabla f)\abs{\nabla f} \\
& = & \frac{n-3}{n-2}C_{1ij}\abs{\nabla f} + \frac{1}{m}C_{1ij}\abs{\nabla f}.
\end{eqnarray*}
It follows that we have $C_{1ij} = 0$ if $m \ne - \frac{n-2}{n-3}$. When $n= 4$, $- \frac{n-2}{n-3} = -2$ which is excluded in the theorem. When $n \geq 5$, an extension of Proposition 5.1 in \cite{caochenBach} shows that $C_{1ij} =0$ for all $m \ne 0, 1$ or $2-n$.
\end{proof}

\begin{proof}[Proof of Theorem \ref{thmWPEBachflatD4}]
From Theorem \ref{thmWPEBachflatallD} we know that $(M^4, g,f)$ has harmonic Weyl tensor and $W(\nabla f, X, Y, Z) = 0$ for any $X, Y, Z \in TM$. We assume that $M$ is not Einstein. At a regular point $p$ of $f$ we assume that the Ricci tensor has distinct eigenvalues. The complement of such points can not contain an open set as $g$ and $f$ are analytic in the harmonic coordinate, see \cite[Proposition 2.8]{HPWlcf}. So it is enough to show that the metric $g$ is locally conformal flat around $p$. Theorem 7.9 in \cite{HPWlcf} says that the metric is locally a warped product over an interval, i.e., $g =\df t^2 + \psi(t)^2 g_L$ where $(L^3, g_L)$ is an Einstein metric and thus has constant curvature. A computation shows that such metric has vanishing Weyl tensors, i.e., it is locally conformally flat.

An alternative approach is to use the symmetries of Weyl tensors to show that they are zeros as in the proof of Theorem 1.1 in \cite{caochenBach}.
\end{proof}

\begin{rem}
In \cite{HPWlcf} the authors considered warped product Einstein manifold with non-empty boundary. Let $w = \exp(-\frac{f}{m})$ in the interior of $M$ and $w = 0$ on the boundary $\partial M$. Both Theorem \ref{thmWPEBachflatD4} and Theorem \ref{thmWPEBachflatallD} can also be extended to case when $M$ has non-empty boundary. For any small $\epsilon > 0$ we define $M_{\epsilon} = \set{x\in M : w(x) \geq \epsilon}$ and we only have to show that $D = 0$ on $M_{\epsilon}$. Then taking the limit $\epsilon \rightarrow 0$ implies that $D = 0$ on $M$. In fact the boundary term of the integral (\ref{eqnIBPBachff}) vanishes:
\begin{equation*}
\int_{\partial M_{\epsilon}} D(\nu, \nabla f, \nabla f) d\vol = 0
\end{equation*}
since the unit normal vector $\nu$ of $\partial M_{\epsilon}$ is parallel to $\nabla f$. So the integral equation (\ref{eqnintegralBachff}) holds on $M_{\epsilon}$ and then $D = 0$ on $M_{\epsilon}$.
\end{rem}

\medskip{}

\vfill

\end{document}